\documentclass{article}

\usepackage[utf8]{inputenc}
\usepackage{amsmath, amssymb, amsthm, bm}
\usepackage{graphicx, booktabs, subfigure}
\usepackage{hyperref}
\usepackage{tikz}
\usepackage{tkz-euclide}
\usetikzlibrary{chains, scopes, positioning, backgrounds, shapes, fit, shadows, calc}
\usepackage{multirow} 
\usepackage{booktabs}
\usepgflibrary{shapes.geometric}
\usepackage{float}
\usepackage{cleveref}
\graphicspath{{figure/}}

\usepackage{algorithmic}
\usepackage{txfonts, cite}

\newtheorem{problem}{Problem}[section]
\newtheorem{theorem}{Theorem}[section]
\usepackage{todonotes}
\graphicspath{{figure/}}

\crefname{hypothesis}{Hypothesis}{Hypotheses}
\numberwithin{equation}{section}

\newtheorem{lemma}{Lemma}[section]
\newtheorem{definition}{Definition}[section]
\theoremstyle{definition}

\graphicspath{{Figures/}}
\usepackage{amsfonts, bm}
\usepackage{graphicx}
\usepackage{epstopdf}
\usepackage{algorithm}
\usepackage{algorithmic}
\usepackage{multirow}
\usepackage{amsmath} 
	
\begin{document}
\title{Inverse source problem of the biharmonic equation from multi-frequency phaseless data}

\author{
	Yan Chang\footnote{School of Mathematics, Harbin Institute of Technology, Harbin, P. R. China. {\it 21B312002@stu.hit.edu.cn} }, \ 
	Yukun Guo\footnote{School of Mathematics, Harbin Institute of Technology, Harbin, P. R. China. {\it ykguo@hit.edu.cn}}, \ and
	Yue Zhao\thanks{School of Mathematics and Statistics, and Key Lab NAA--MOE, Central China Normal University,
		Wuhan 430079, China.
		{\it zhaoy@ccnu.edu.cn} (Corresponding author)}
	}
	\date{}
	\maketitle
	
\begin{abstract}
	This work deals with an inverse source problem for the biharmonic wave equation.  A two-stage numerical method is proposed to identify the unknown source from the multi-frequency phaseless data. In the first stage, we introduce some artificially auxiliary point sources to the inverse source system and establish a phase retrieval formula. Theoretically, we point out that the phase can be uniquely determined and estimate the stability of this phase retrieval approach. Once the phase information is retrieved, the Fourier method is adopted to reconstruct the source function from the phased multi-frequency data. The proposed method is easy-to-implement and there is no forward solver involved in the reconstruction. Numerical experiments are conducted to verify the performance of the proposed method.
\end{abstract}
	
{\textbf{Keywords:} 
	Biharmonic equation, inverse source problem, phaseless data, phase retrieval, reference source, Fourier method
}

\section{Introduction}

The inverse problem of reconstructing the source from the measurements has wide applications in various scenarios varying from antenna synthesis, acoustic tomography, medical imaging, telecommunication, and many other fields (see, e.g., \cite{BEEH, Deng,  Leone, Stefanov, Thio} and the reference therein). 

Recently, much attention has been devoted to investigating the numerical methods for the inverse source problem. We refer to \cite{Alvex, Bao, LXD} for the iterative method, the recursive method, and the direct sampling method to reconstruct the acoustic source. As the method of interest, in recent years, the Fourier method has been paid enduring attention to and shows powerful ability in reconstructing the square-integrable source function independent of the wave number from the multi-frequency measurements. In \cite{IP15}, a Fourier method is proposed to reconstruct the source in the Helmholtz equation from multi-frequency near-field measurements. Further, this method is applied to various physical models such as the far-field case \cite{WXCfar}, the electromagnetic wave \cite{WG}, the elastic wave \cite{elasticFourier}, and the biharmonic case \cite{CGYZ23}. A major advantage of this method is that the method is based on the Fourier expansion of the source function and the Fourier coefficients of the expansion can be directly computed accurately from the multi-frequency measurements.  Since there is no forward solver involved in the reconstruction process, this method is easy to implement and direct.

However, there are wide practical situations where the measurements of full complex-valued data cannot be obtained, while only the intensity is available (see, e.g., \cite{9,11,47,48}). In such a scenario, the numerical method utilizing the full information fails to work well, and thus, it is necessary to develop an efficient phase retrieval method to retrieve the phase information such that the source can be reconstructed from the recovered radiated field. In recent years, there have been considerable studies on such problems. For example, Bao et. al \cite{Bao1} propose a continuation scheme to reconstruct the source from phaseless measurements on the boundary. In \cite{IP18}, Zhang et. al propose a phase retrieval technique by introducing the reference source points to the inverse source system and they utilize the Fourier method \cite{IP15} to recover the source from multi-frequency data. We also refer to \cite{Agaltsov, Dong, LYZZ21, IPI23} for the relevant numerical methods for the inverse scattering problem with phaseless data.

Arising from the modeling of elasticity, the biharmonic operator has wide applications ranging from offshore runway design, seismic cloaks, and the platonic crystal (see \cite{8, 22, 27, 30} and the references therein). Nevertheless, compared with the aforementioned second-order wave equation,  the inverse scattering problems for the fourth-order biharmonic wave equations are significantly less investigated. A reason accounting for this is that the increase of the order leads to the invalidity of the methods developed for the second-order equation. In addition, the solution of the higher-order equation admits more sophisticated properties, which may lead to the failure of the conventional computational methods.

In the current work, we are concerned with the inverse problem of recovering the source function in the biharmonic equation from the phaseless radiated field. The overall idea is divided into two stages. In the first stage, we design a phase retrieval technique to retrieve the phase from the phaseless data by introducing artificial point sources to the inverse source system. Then, the phaseless problem can be translated into its phased version with the reconstructed phase. During this stage, it is technical to choose the proper point sources as well as several parameters such that the absolute value of the determinants of the coefficient matrices admit positive lower bounds. These invertible matrices further guarantee the stability of the phase-retrieval system. Compared with its acoustic counterpart \cite{IP18}, in the present biharmonic case, the incorporation of the modified Bessel function brings many substantial difficulties in the analysis, especially for the stability estimates. Therefore, it is of significance to investigate the novel asymptotic behaviors of the modified Bessel function, which is one of the contributions of our paper.

In the second stage, the Fourier method can be adopted to reconstruct the source from the radiated field with recovered phase information. Note that the applicability of the Fourier method for identifying the biharmonic sources is studied in \cite{CGYZ23}. Motivated by \cite{CGYZ23, IP18}, we aim to extend the methodology to this new model of determining the biharmonic source from phaseless data. In addition, we would like to point out that, this framework of source recovery performs stably with the noise-corrupted measurements. Numerical experiments demonstrate that the combination of the phase retrieval technique and the Fourier method exhibits excellent capability in reconstructing the source function from the phaseless radiating data.

The rest of this paper is arranged as follows: In the next section, we introduce several reference source points to the inverse source system and develop an explicit formula to retrieve the phase information. Regarding the phase retrieval method, we analyze the stability in \Cref{sec: stability}. In \Cref{sec: Fourier}, we adopt the Fourier method to reconstruct the source from the retrieved information and finally verify the performance of the proposed method in \Cref{sec: numerical}.

\section{Preliminaries}\label{sec: Fourier}

We shall first give a mathematical formulation of the multifrequency inverse source problem under consideration. Then, the Fourier method for solving the inverse source problem will be briefly described.

\subsection{Model problem}
In this paper, we consider the inverse problem of reconstructing an unknown source in the biharmonic equation from the multi-frequency phaseless data. 

$$ 
$$

Let $a>0$ and set $V_0:=\left(-a, a\right)\times\left(-a, a\right)$. Suppose that $S\in L^2(\mathbb{R}^2)$ is the source function (independent of the wave number $k>0$) such that $\mathrm{supp} S\subset V_0$. The propagation of the biharmonic wave $u$ can be modeled by the fourth-order equation
\begin{align}\label{eq: biharmonic}
	\Delta^2u-k^4u=S\quad\text{in }\mathbb{R}^2.
\end{align}

To ensure that the wave field is outgoing, the Sommerfeld radiation conditions are imposed on $u$ and $\Delta u$ as follows:
\begin{align}\label{eq:Sommerfeld}
	\lim\limits_{r\to\infty}\sqrt{r}(\partial_ru-\mathrm{i}ku)=0,\quad\lim\limits_{r\to\infty}\sqrt{r}(\partial_r\Delta u-\mathrm{i}k\Delta u)=0, \quad r=|x|,
\end{align}
which hold uniformly in all directions. 

Let $\Phi_k(x, y)$ be the fundamental solution to the biharmonic wave equation, i.e.,
\begin{align}\label{eq:Phi}
	\Phi_k(x, y)=\frac{\mathrm{i}}{8k^2}\left(H_0^{(1)}(k|x-y|)-H_0^{(1)}(\mathrm{i}k|x-y|)\right),
\end{align}
where $H_0^{(1)}$ is the Hankel function of the first kind and order zero. Then it is known that there exists a unique solution $u$ to \cref{eq: biharmonic} and \cref{eq:Sommerfeld}, explicitly, 
\begin{align}\label{eq: solution}
	u(x, k)=\int_{V_0}\Phi_k(x, y)S(y)\mathrm{d}y.
\end{align}

Choose some acquisition curve $\Gamma_R:=\{x\in\mathbb{R}^2:|x|=R\}$ such that $V_0\subset B_R:=\{x\in\mathbb{R}^2:|x|<R\}.$ Let $N\in\mathbb{N}_+$ and $\mathbb{K}_N$ be an admissible set consisting of a finite number of wave numbers, then the inverse problem under consideration can be formulated by:
\begin{problem}[Phaseless multifrequency inverse source problem]\label{pro: problem}
	Recover $S(x)$ from the multi-frequency phaseless data 
	$$
	\mathbb{D}_{\rm{PL}}:=\left\{|u(x, k)|,\ |\Delta u(x, k)|:x\in\Gamma_R,\ k\in\mathbb{K}_N\right\}.
	$$
\end{problem}

It is obvious that both $S$ and $-S$ produce the identical intensities
$$
\left|\int_{V_0}\Phi_k(x, y)S(y)\mathrm{d}y\right|,\quad\left|\int_{V_0}\Delta_x\Phi_k(x, y)S(y)\mathrm{d}y\right|,\quad\forall k\in\mathbb{R}_+,\ x\in\mathbb{R}^2,
$$
which illustrates that the source $S$ can not be uniquely determined from the phaseless data $\mathbb{D}_{\rm PL}.$ 
Thus, to tackle this dilemma, we develop a two-stage numerical method to deal with Problem \ref{pro: problem} by reformulating it into the following two inverse problems:

\begin{problem}[Phase retrieval]\label{pro:phaseRetrieval}
	Reconstruct the phase information from 
	$\mathbb{D}_{\rm{PL}}$ to obtain the radiated field data 
	$$\mathbb{D}_{\rm{P}}=
	\{u(x, k),\Delta u(x, k):x\in\Gamma_R,\ k\in\mathbb{K}_N\}.
	$$ 
\end{problem}

\begin{problem}[Multifrequency inverse source problem]\label{pro: ISP} 
	Reconstruct the source function $S$ from $\mathbb{D}_{\rm{P}}$.
\end{problem}

In the practical implementation, we first focus on \cref{pro:phaseRetrieval} to retrieve the phase information and then turn to Problem \ref{pro: ISP} to recover the source function.  Nevertheless, the aim of \cref{pro:phaseRetrieval} is to provide full data for \cref{pro: ISP}. Thus, we first give a brief description of the Fourier method for solving \cref{pro: ISP} under the assumption that $\mathbb{D}_{\rm{P}}$ is available. \Cref{pro:phaseRetrieval} will be considered in the next section.

\subsection{Fourier method}

We recall here the key rudiments of the Fourier method (see \cite{CGYZ23, IP15} for more details). The basic idea is to approximate the source function $S$ through the following Fourier expansion 
\begin{align}\label{eq:SN}
	S_N(x):=\sum_{|\bm{l}|_\infty\le N}\hat{s}_{\bm{l}}\phi_{\bm l}(x), \quad {\bm l}\in\mathbb{Z}^2,
\end{align}
where $\hat{s}_{\bm l}$ are the Fourier coefficients, and 
\[
\phi_{\bm l}(x)=\mathrm{e}^{\mathrm{i}\frac{\pi}{a}{\bm l}\cdot x},\quad {\bm l}\in\mathbb{Z}^2,
\]
is the Fourier basis function. Once $\hat{s}_{\bm l}$ is determined by $\mathbb{D}_{\rm{P}},$ the reconstructed source function $S_N$ can be obtained correspondingly. As a result, in the rest of this subsection, we are devoted to computing $\hat{s}_{\bm l}$. To start with, we recall the definition of admissible frequencies (see, \cite{CGYZ23}).

\begin{definition}[Admissible wave numbers,]\label{def:k}
	For $N\in\mathbb{N}_+,$ we define the admissible wave numbers as follows:
	$$
	\mathbb{K}_N:=\left\{\frac{\pi}{a}|\bm l|:{\bm l}\in\mathbb{Z}^2,1\le|\bm l|_\infty\le N\right\}\cup\{k_0\},
	$$
	where $k_0>0$ is a small wave number satisfying $0<k_0<1/R.$
\end{definition}

In the following, we present the formulae to compute the Fourier coefficients $\hat{s}_{\bm l},\,|\bm l|_\infty\le N.$

Let $\nu_\rho$ be the unit outward normal to $\Gamma_\rho:=\{x\in\mathbb{R}^2:|x|=\rho\}$ with $\rho>R.$ Given $\mathbb{D}_{\rm{P}}$, we introduce two auxiliary functions
$$
u_H=-\frac{1}{2k^2}\left(\Delta u-k^2u\right),\quad u_M=\frac{1}{2k^2}\left(\Delta u+k^2u\right),
$$
which satisfy
$$
u=u_H+u_M,\quad\Delta u=k^2(u_M-u_H).
$$

Under the polar coordinate $(r, \theta):x=r(\cos\theta, \sin\theta)$, we consider the following series expansions for $x\in\Gamma_\rho:$
\begin{align*}
	&w_H(x, k)=\sum_{n\in\mathbb{Z}}\frac{H_n^{(1)}(k\rho)}{H_n^{(1)}(kR)}\hat{u}_{k, n}^{H}\mathrm{e}^{\mathrm{i}n\theta},\quad 
	w_M(x, k)=\sum_{n\in\mathbb{Z}}\frac{K_n(k\rho)}{K_n(kR)}\hat{u}_{k, n}^{M}\mathrm{e}^{\mathrm{i}n\theta},\\
	&\partial_{\nu_\rho}w_H(x, k)=\sum_{n\in\mathbb{Z}}k\frac{{H_n^{(1)}}'(k\rho)}{H_n^{(1)}(kR)}\hat{u}_{k, n}^{H}\mathrm{e}^{\mathrm{i}n\theta},\quad 
	\partial_{\nu_\rho}w_M(x, k)=\sum_{n\in\mathbb{Z}}k\frac{{K'_n}(k\rho)}{K_n(kR)}\hat{u}_{k, n}^{M}\mathrm{e}^{\mathrm{i}n\theta},
\end{align*}
where $H_n^{(1)}$ and $K_n$ are the first-kind Hankel function and second-kind modified Bessel function of order $n,$ respectively, and the Fourier coefficients $\hat{u}_{k, n}^{\xi},\,\xi\in\{H, M\}$ can be computed by
\[
\hat{u}_{k, n}^{\xi}=\frac{1}{2\pi}\int_0^{2\pi}u_\xi(R,\theta; k)\mathrm{e}^{-\mathrm{i}n\theta}\,\mathrm{d}\theta.
\]

By letting $\lambda\in(0,\frac{a}{2\pi}), {\bm{l}_0}=(\lambda,0), k_0=\pi\lambda/a$, and
\begin{align*}
	w&=w_H+w_M, \quad \partial_{\nu_\rho}w=\partial_{\nu_\rho}w_H+\partial_{\nu_\rho}w_M,\\
	\Delta w&=k^2(w_M-w_H), \quad \partial_{\nu_\rho}\Delta w=k^2(\partial_{\nu_\rho}w_M-\partial_{\nu_\rho}w_H),
\end{align*}
the Fourier coefficients can be calculated via the following integral identities:                                      
\begin{align}\nonumber
	\hat{s}_{\bm l}&=\frac{1}{a^2}\int_{\Gamma_\rho}\left(\partial_{\nu_\rho}\Delta w(x, k)+\mathrm{i}\frac{2\pi}{a}({\bm l}\cdot\nu_\rho)\Delta w(x, k)\right)\overline{\phi_{\bm l}(x)}\,\mathrm{d}s(x)\\
	&\quad-\frac{k^2}{a^2}\int_{\Gamma_\rho}\left(\partial_{\nu_\rho}w(x, k)+\mathrm{i}\frac{2\pi}{a}({\bm l}\cdot \nu_\rho)w(x, k)\right)\overline{\phi_{\bm l}(x)}\,\mathrm{d}s(x), \quad k\in\mathbb{K}_N\backslash\{k_0\},\label{eq:sl}
\end{align}
and
\begin{align}\nonumber
	\hat{s}_{\bm{l}_0} & =\frac{1}{a^2\mathrm{sinc}(\lambda\pi)}\int_{\Gamma_\rho}\left(\partial_{\nu_\rho}\Delta w(x, k_0)+\mathrm{i}\frac{2\pi}{a}({\bm{l}}_0\cdot{\nu_\rho})\Delta w(x, k_0)\right)\overline{ \phi_{\bm{l}_0}(x)}\,\mathrm{d} s(x) \\
	\nonumber
	&\quad -\frac{k_0^2}{a^2\mathrm{sinc}(\lambda\pi)}\int_{\Gamma_\rho}\left(\partial_{\nu_\rho} w(x, k_0)+\mathrm{i}\frac{2\pi}{a}({\bm{l}}_0\cdot{\nu_\rho}) w(x, k_0)\right)\overline{ \phi_{\bm{l}_0}(x)}\,\mathrm{d} s(x) \\
	&\quad -\frac{1}{a^2\mathrm{sinc}(\lambda\pi)}\sum_{1\le|{\bm l}|_\infty\le N}\hat{s}_{\bm l}\int_{V_0} \phi_{\bm l}(x)\overline{ \phi_{\bm{l}_0}(x)}\,\mathrm{d} x,
	\label{eq:s0}
\end{align}
where $\mathrm{sinc}(x)=\sin x/x$.

By taking the ill-posedness of the inverse source problem into account, we formulate the Fourier method in Algorithm \ref{alg: Fourier} from the noisy data $\mathbb{D}_{\rm{P}}^\varepsilon,$ by replacing the radiated field data $u(x, k)$ and $\Delta u(x, k)$ in $\mathbb{D}_{\rm{P}}$ with their noisy measurements $u^\varepsilon(x, k)$ and $\Delta u^\varepsilon(x, k).$ Here and in what follows, we use the  superscript ``$\varepsilon$'' to indicate the dependence on the noise level $\varepsilon.$ We refer to \cite{CGYZ23} for further implementational details.
\begin{algorithm}
	\begin{algorithmic}
		\STATE{\textbf{Step~1} Choose the parameters $\lambda,\rho, N$ and the admissible frequencies $\mathbb{K}_N$;}\\
		\STATE{\textbf{Step~2} Collect/evaluate the multi-frequency noisy data $\mathbb{D}_{\rm{P}}^\varepsilon;$}\\
		\STATE{\textbf{Step~3} Evaluate the Cauchy data $u^\varepsilon,$ $\partial_{\nu_\rho}u^\varepsilon,$ $\Delta u^\varepsilon$ and $\partial_{\nu_\rho}\Delta u^\varepsilon$ on $\Gamma_\rho.$}\\
		\STATE{\textbf{Step~4} Compute the Fourier coefficients $\hat{s}^\varepsilon_{\bm{l}_0}$ and $\hat{s}_{\bm l}^\varepsilon(1\le|\bm l|_\infty\le N)$ by \cref{eq:sl}--\cref{eq:s0};}\\
		\STATE{\textbf{Step~5} Define the truncated Fourier expansion as \cref{eq:SN} by taking $\hat{s}_{\bm l}$ to be $\hat{s}_{\bm l}^\varepsilon,$ for $|\bm l|_\infty\le N$;}\\
		\STATE{\textbf{Step~6}  We take the source $S_N^\varepsilon$ as an approximate to $S.$}
		\caption{Fourier method to reconstruct $S.$}\label{alg: Fourier}
	\end{algorithmic}
\end{algorithm}

\section{Phase retrieval by auxiliary sources}

In this section, we shall develop a numerical method to retrieve the phase by introducing several reference sources to the inverse source setup. The core idea of this technique is to incorporate auxiliary point sources for each receiver to supplement the missing information. It deserves noting that, if only a single extra point source is added to the inversion system,  then the radiated fields emitted respectively by the unknown source and the known auxiliary source are still coupled intricately in the phaseless measurements. Magically, this difficulty can be overcome skillfully once two auxiliary source points are appropriately integrated into the system. Meanwhile, the missing phase information can be also recovered accurately provided the measurement noise is sufficiently small.

To elaborate on the phase retrieval technique, we first define several notations and definitions. Given $m>0,$ we split $B_R$ and $\Gamma_R$ into $m$ parts, i.e.,
\begin{align*}
	B_R=\mathop{\cup}_{j=1}^mB_j,\quad\Gamma_R=\mathop{\cup}_{j=1}^m\Gamma_j,
\end{align*}
with
\begin{align*}
	&B_j:=\{r(\cos\vartheta,\sin\vartheta):0\le r\le R,\vartheta_{2j-2}\le\vartheta\le\vartheta_{2j}\},\quad j=1,\cdots, m,\\
	&\Gamma_j:=\{R(\cos\vartheta,\sin\vartheta):\vartheta_{2j-2}\le\vartheta\le\vartheta_{2j}\},\quad  j=1,\cdots, m,
\end{align*}
and $\vartheta_j=j\pi/m$ for $j\in\mathbb{N}_+.$

For each $j(j=1,\cdots, m),$ we choose two real constants $\lambda_{j,\ell}$ such that $\sqrt{2}a/R\le|\lambda_{j,\ell}|<1$ for $\ell=1,2,$  and define two points $z_{j,\ell}:=\lambda_{j,\ell}R(\cos\vartheta_{2j-1},\sin\vartheta_{2j-1}).$  
To delicately balance the intensities between the target source and the artificial source, we also need the following scaling factors
$$	
c_{j,\ell, k}:=\frac{\|u(\cdot, k)\|_{j,\infty}}{\|\Phi_k(\cdot, z_{j,\ell})\|_{j,\infty}},\quad j=1,\cdots, m,\ \ell=1,2,
$$
with $\|\cdot\|_{j,\infty}:=\|\cdot\|_{L^\infty(\Gamma_j)}$. It can been seen that $\Psi_{j,\ell}(x, k):=-c_{j,\ell, k}\Phi_k(x, z_{j,\ell})$ satisfies the following inhomogeneous biharmonic equation
$$
\Delta^2\Psi_{j,\ell}-k^4\Psi_{j,\ell}=c_{j,\ell, k}\delta_{j,\ell},\quad \text{in }\mathbb{R}^2,
$$
where $\delta$ designates the Dirac distributions. Further, from the linearity of the biharmonic equation, one easily deduces that $v_{j,\ell}=u+\Psi_{j,\ell}$ is the unique solution to the problem
$$
\begin{cases}
	\Delta^2v_{j,\ell}-k^4v_{j,\ell}=S+c_{j,\ell, k}\delta_{j,\ell}, \quad \text{in }\mathbb{R}^2,\\
	\lim\limits_{r=|x|\to\infty}\sqrt{r}\left(\dfrac{\partial v_{j,\ell}}{\partial r}-\mathrm{i}k v_{j,\ell}\right)=0,\vspace{2mm}\\
	\lim\limits_{r=|x|\to\infty}\sqrt{r}\left(\dfrac{\partial \Delta v_{j,\ell}}{\partial r}-\mathrm{i}k \Delta v_{j,\ell}\right)=0.
\end{cases}
$$

We are now concerned with \cref{pro:phaseRetrieval}. For ease of notation, we denote $u_{j, k}(\cdot):=\left.u(\cdot, k)\right|_{\Gamma_j},$
for each $j=1,\cdots, m,$ and $k\in\mathbb{K}_N.$ Similarly, we also write $v_{j,\ell, k}(\cdot)=v_{j,\ell}(\cdot, k)|_{\Gamma_j}$ to indicate the dependence of $v_{j,\ell}$ on $k$.

Given the phaseless data $\left\{|\square u_{j, k}|,\,|\square v_{j,\ell, k}|:\square\in\{\mathcal{I},\Delta\},\ \ell=1,2, \ k\in\mathbb{K}_N\right\},$  we derive that
\begin{align}
	\label{eq:ujk} & \left(\Re \square u_{j, k}\right)^2+\left(\Im \square u_{j, k}\right)^2=|\square u_{j, k}|^2,\\
	\label{eq:vj1k} & \left(\Re \square u_{j, k}-c_{j,1,k}\Re \square \Phi_{j,1,k}\right)^2+
	\left(\Im \square u_{j, k}-c_{j,1,k}\Im \square \Phi_{j, 1, k}\right)=|\square v_{j,1,k}|^2,\\
	\label{eq:vj2k} & \left(\Re \square u_{j, k}-c_{j,2,k}\Re \square \Phi_{j,2,k}\right)^2+
	\left(\Im \square u_{j, k}-c_{j,2,k}\Im \square \Phi_{j, 2, k}\right)=|\square v_{j,2,k}|^2,
\end{align}
with $\square$ denoting the identity operator $\mathcal{I}$ or the Laplace operator $\Delta,$ depending on the data ($u(x, k)$ or $\Delta u(x, k)$) to be retrieved. In what follows, we may omit the identity operator $\mathcal{I}$ for the notational simplification in the case $\square=\mathcal{I}$.

For the coefficients in \cref{eq:vj1k} and \cref{eq:vj2k}, we compute from \cref{eq:Phi} explicitly that
\begin{align*}
	&\Re\Phi_k(x, y)= -\frac{1}{8k^2}\left(Y_0(k|x-y|)+H_0(k|x-y|) \right),\\
	&\Re\Delta_x\Phi_k(x, y)= \frac{1}{8}\left(Y_0(k|x-y|)-H_0(k|x-y|) \right),\\
	&\Im\Phi_k(x, y)=\frac{1}{8k^2}J_0(k|x-y|),\quad \Im\Delta_x\Phi_k(x, y)=-\frac{1}{8}J_0(k|x-y|),
\end{align*}
where we have denoted $H_0(t)=\mathrm{i}H_0^{(1)}(\mathrm{i}t).$
Subtracting \cref{eq:ujk} from  \cref{eq:vj1k} and \cref{eq:vj2k}, respectively, we get the following linear system concerning $\Re \square v_{j, k}$ and $\Im \square v_{j, k}$:
\begin{align}\label{eq: equation}
	\left\{
	\begin{aligned}
		A_{j,1,k}^\square\Re\square u_{j, k}+B_{j,1,k}^\square\Im\square v_{j, k}=f_{j,1,k}^\square,\\
		A_{j,2,k}^\square\Re\square u_{j, k}+B_{j,2,k}^\square\Im\square v_{j, k}=f_{j,2,k}^\square,
	\end{aligned}
	\right.
\end{align}
where, for $\ell=1,2,$
\begin{align}
	\label{eq:A} A_{j,\ell, k}^{\square} &=
	\begin{cases}
		\frac{1}{k^2}\left(Y_0(k r_{j,\ell})+H_0(k r_{j,\ell})\right), & \text{if}\ \square = \mathcal{I},\\
		H_0(k r_{j,\ell})-Y_0(k r_{j,\ell}), & \text{if}\ \square = \Delta,
	\end{cases} \\
	\label{eq:B} B_{j,\ell, k}^{\square} & =
	\begin{cases}
		-\frac{1}{k^2}J_0(k r_{j,\ell}),&  \text{if}\ \square = \mathcal{I},\\
		J_0(k r_{j,\ell}),&  \text{if}\ \square = \Delta,
	\end{cases}\\
	\label{eq:f} f_{j,\ell, k}^\square & =\frac{4}{c_{j,\ell, k}}\left(\left|\square v_{j,\ell, k}\right|^2-\left|\square u_{j, k}\right|^2\right)-\frac{c_{j,\ell, k}}{16}\left(\left|A_{j,\ell, k}^\square\right|^2+\left|B_{j,\ell, k}^\square\right|^2\right).
\end{align}
with $r_{j,\ell}=|x-z_{j,\ell}|,\ x\in\Gamma_j.$ 

Now a straightforward calculation retrieves the phased data:
\begin{equation}\label{eq:retrieve}
	\square u_{j, k}=\frac{\det D_{j, k}^{R,\square}}{\det D_{j, k}^\square}+\mathrm{i}\frac{\det D_{j, k}^{I,\square}}{\det D_{j, k}^\square},
\end{equation}
where
\begin{align*}
	D_{j, k}^\square=
	\begin{bmatrix}\vspace{+.2cm}
		A_{j,1,k}^\square & B_{j,1,k}^\square \\
		A_{j,2,k}^\square & B_{j,2,k}^\square
	\end{bmatrix},\quad
	D_{j, k}^{R,\square}=
	\begin{bmatrix}\vspace{+.2cm}
		f_{j,1,k}^\square & B_{j,1,k}^\square \\
		f_{j,2,k}^\square & B_{j,2,k}^\square
	\end{bmatrix},\quad
	D_{j, k}^{I,\square}=
	\begin{bmatrix}\vspace{+.2cm}
		A_{j,1,k}^\square & f_{j,1,k}^\square \\
		A_{j,2,k}^\square & f_{j,2,k}^\square
	\end{bmatrix}.
\end{align*}

To recover $\square u_{j, k}$ from the measurements $|\square u_{j, k}|,\,|\square v_{j,1,k}|,$ and  $|\square v_{j,2,k}|,$ we only need to solve \eqref{eq: equation} once  \eqref{eq: equation} is uniquely solvable. For this, the parameters $R, m$ and $R_{j,\ell}$ for $\ell=1,2$ are supposed to be chosen properly. In this paper, we take these parameters as follows:
\begin{equation}\label{eq: parameter}
	\begin{cases}
		m\ge 10,\quad\lambda_{j,1}=\dfrac{1}{2},\quad k_0=\dfrac{\pi}{30a},\quad \tau\ge 6,\vspace{2mm} \\ 
		\begin{cases}
			R=\tau a,\quad\lambda_{j,2}=\dfrac{1}{2}+\dfrac{\pi}{2kR},& \text{if }\ \ k\in\mathbb{K}_N\backslash\{k_0\},\vspace{2mm}\\
			R=6a,\quad\lambda_{j,2}=-\dfrac{3}{2},&\text{if }\ \ k=k_0.
		\end{cases}
	\end{cases}
\end{equation}

As shall be seen in the next section, under such choice of the parameters \cref{eq: parameter}, the equation system \cref{eq: equation} is uniquely solvable. Since the practically measured data is always polluted by noise, we finally formulate the phase retrieval algorithm with perturbed data in \cref{alg: PR}.

\begin{algorithm}
	\begin{algorithmic}
		\STATE{\textbf{Step~1} Taking the parameters $R, m$ and $\lambda_{j,\ell}$ for $j=1,\cdots, m,\ell=1,2$ as in \eqref{eq: parameter};}\\
		\STATE{\textbf{Step~2} Acquist the noisy phaseless data $\mathbb{D}_{\rm PL}^\varepsilon$ and compute the scaling factors $c_{j,\ell, k}^\varepsilon$ for $j=1,\cdots, m,\ell=1,2,$ and $k\in\mathbb{K}_N;$}\\
		\textbf{Step~3} For each $j=1,\cdots, m$ and $k\in\mathbb{K}_N$:\\
		\qquad\textbf{Step~3.1} Place two point sources $c_{j,\ell, k}^\varepsilon\delta_{j,\ell}$ to the inverse source system; \\
		\qquad\textbf{Step~3.2} For $\ell=1,2$, collect the phaseless data $|v_{j,\ell}^\varepsilon(x;k)|$, $|\Delta v_{j,\ell}^\varepsilon(x;k)|$ on $\Gamma_j$; \\
		\textbf{Step~4} Retrieve the phase by \cref{eq:retrieve} to obtain the radiated field $\{u^\varepsilon(x, k),\Delta u^\varepsilon(x, k):x\in\Gamma_j, k\in\mathbb{K}_N\}.$ 
		\caption{Phase retrieval algorithm for \cref{pro:phaseRetrieval}}\label{alg: PR}
	\end{algorithmic}
\end{algorithm}

\section{Stability of the phase retrieval}\label{sec: stability}

In this section, we are devoted to analyzing the stability of the phase retrieval algorithm proposed in Algorithm \ref{alg: PR}. 
We start with an estimate on the modified Bessel function $K_0$, which will play a key role in the later analysis.

\begin{lemma}\label{lem: K}
	For $x>0,$ it holds that
	$$
	\left|K_0(x)-\sqrt{\frac{\pi}{2x}}\mathrm{e}^{-x}\right|\le\frac{\sqrt{\pi}x^{-3/2}}{8\sqrt{2}}.
	$$
\end{lemma}

\begin{proof}
	From the integral representation of $K_0(x)$ \cite[Chapter~14(14.131)]{Arfken} under the assumption $x>0,$ we know 
	\begin{equation}\label{eq:K0integral}
		K_0(x)=\sqrt{\frac{\pi}{2x}}\frac{\mathrm{e}^{-x}}{\Gamma(\frac{1}{2})}\int_0^\infty\mathrm{e}^{-t}t^{-1/2}\left(1+\frac{t}{2x}\right)^{-1/2}\mathrm{d}t,
	\end{equation}
	where $\Gamma(\cdot)$ is the Gamma function 
	$$
	\Gamma(t):=\int_0^\infty\mathrm{e}^{-s}s^{t-1}\,\mathrm{d}s.
	$$
	
	Because
	$$
	\left(1+\frac{t}{2x}\right)^{-1/2}=1-\frac{t}{4x}\int_0^1\left(1+\frac{t s}{2x}\right)^{-3/2}\mathrm{d}s,
	$$
	we rewrite \eqref{eq:K0integral} as
	\begin{align*}
		K_0(x)=\sqrt{\frac{\pi}{2x}}\mathrm{e}^{-x}-\sqrt{\frac{\pi}{2x}}\frac{\mathrm{e}^{-x}}{\Gamma(\frac{1}{2})}\int_0^\infty\mathrm{e}^{-t}\frac{t^{1/2}}{4x}\,\mathrm{d}t\int_0^1\left(1+\frac{t s}{2x}\right)^{-3/2}\mathrm{d}s.
	\end{align*}
	
	Since $t\ge 0, x>0,$ it holds that
	\begin{align*}
		\left|K_0(x)-\sqrt{\frac{\pi}{2x}}\mathrm{e}^{-x}\right| & \le\sqrt{\frac{\pi}{2x}}\frac{\mathrm{e}^{-x}}{4\Gamma(\frac{1}{2})x}\int_0^\infty\mathrm{e}^{-t}{t^{1/2}}\,\mathrm{d}t\int_0^1\left|1+\frac{t s}{2x}\right|^{-3/2}\mathrm{d}s \\
		&\le\sqrt{\frac{\pi}{2x}}\frac{\mathrm{e}^{-x}}{4\Gamma(\frac{1}{2})x}\Gamma\left(\frac{3}{2}\right)\\
		&\le\sqrt{\frac{\pi}{2x}}\frac{\mathrm{e}^{-x}}{8\Gamma(\frac{1}{2})x}\Gamma\left(\frac{1}{2}\right)\\
		&\le\frac{\sqrt{\pi}}{8\sqrt{2}}x^{-3/2}\mathrm{e}^{-x}\\
		&\le\frac{\sqrt{\pi}x^{-3/2}}{8\sqrt{2}},
	\end{align*}
	which completes the proof.
\end{proof}

\begin{lemma}\label{lem:H01}\cite{IP18}
	Let $x>0,$ then the following estimate holds
	$$
	\left|H_0^{(1)}(x)-\sqrt{\frac{2}{\pi x}}\mathrm{e}^{\mathrm{i}(x-\pi/4)}\right|\le\frac{x^{-3/2}}{4\sqrt{2\pi}}.
	$$
\end{lemma}

In terms of \Cref{lem: K} and $\pi H_0(x)=2 K_0(x)$, we immediately have 
\begin{equation}\label{eq:errorH0}
	\left|H_0(x)-\sqrt{\frac{2}{\pi x}}\mathrm{e}^{-x}\right|\le\frac{x^{-3/2}}{4\sqrt{2\pi}},\quad \forall x>0.
\end{equation}

Based on the above two lemmas, we are now ready to give an estimate on $|\det(D^\square_{j, k})|,$ which characterizes the invertibility of  
matrices $D^\square_{j, k}$.
\begin{theorem}
	Choose the parameters as defined in \eqref{eq: parameter} and let
	$$
	p=\begin{cases}
		5, & \text{if }\ \square = \mathcal{I}, \\
		1, & \text{if }\ \square = \Delta,
	\end{cases}
	$$ 
	then there exists a constant $C>0$, such that
	\begin{equation}\label{eq: bound}
		\left|\det \left(D_{j, k}^\square\right)\right|\ge C k^{-p}, \quad \forall k\in \mathbb{K}_N.
	\end{equation}
\end{theorem}

\begin{proof}
	We only prove the estimate for $\left|\det(D_{j, k})\right|.$ The estimate on $|\det(D_{j, k}^\Delta)|$ can be obtained by a similar argument. The proof is divided into two cases.
	
	Case (i): $k\in\mathbb{K}_N\backslash\{k_0\}.$
	In this case, 
	$$ k\ge k_{\min}:=\min\limits_{k\in\mathbb{K}_N\backslash\{k_0\}}k=\frac{\pi}{a},$$
	
	and the parameters \cref{eq: parameter} imply that \cite{IP18}
	\begin{align}\label{eq:3.7}
		&\lambda_{j,2}\le\frac{2}{3},\quad k r_{j,1}\ge\frac{\tau\pi}{2},\quad k r_{j,2}>\frac{\tau\pi}{3},\\
		&0.3363\pi\le k(r_{j,1}-r_{j,2})\le \frac{\pi}{2}.\label{eq:3.9}
	\end{align}
	Denote 
	\begin{align}\label{eq:alphaJY}
		&\alpha_J(x):=J_0(x)-\sqrt{\frac{2}{\pi x}}\cos\left(x-\frac{\pi}{4}\right),\quad\alpha_Y(x):=Y_0(x)-\sqrt{\frac{2}{\pi x}}\sin\left(x-\frac{\pi}{4}\right),\\
		&\alpha_H(x):=H_0(x)-\sqrt{\frac{2}{\pi x}}\mathrm{e}^{-x}.\label{eq:alphaH}
	\end{align}
	From \Cref{lem:H01} and \eqref{eq:errorH0}, we know that
	\begin{align}\label{eq:alpha_l_bound}
		\left|\alpha_\ell(x)\right| & \le\frac{x^{-3/2}}{4\sqrt{2\pi}}, \quad\ell\in\{J,H,Y\},\quad\forall x>0.
	\end{align}
	
	Combining \eqref{eq:retrieve}, \eqref{eq:alphaJY}, \eqref{eq:alphaH} gives
	\begin{align*}
		\det \left(D_{j, k}\right) & =A_{j,1,k}B_{j,2,k}-A_{j,2,k}B_{j,1,k}\\
		&=\frac{1}{k^4}(J_0(s_1)(Y_0(s_2)+H_0(s_2))-J_0(s_2)(Y_0(s_1)+H_0(s_1)))\\
		&=\frac{2}{\pi k^5\sqrt{r_{j, 1}r_{j, 2}}}\sin(s_2-s_1)+\gamma(s_1, s_2),
	\end{align*}
	where $s_\ell=k r_{j,\ell}(\ell=1,2),$ and
	\begin{align*}
		\gamma(s_1,s_2)&=\frac{1}{k^4}\left\{
		\frac{2}{\pi\sqrt{s_1s_2}}\left(\mathrm{e}^{-s_2}\cos\left(s_1-\frac{\pi}{4}\right)-\mathrm{e}^{-s_1}\cos\left(s_2-\frac{\pi}{4}\right)\right)\right.\\
		&\quad\left.+\alpha_J(s_1)(\alpha_Y(s_2)+\alpha_H(s_2))-\alpha_J(s_2)(\alpha_Y(s_1)+\alpha_H(s_1))\right.\\
		&\quad\left.+\sqrt{\frac{2}{\pi s_2}}\alpha_J(s_1)\left(\sin\left(s_2-\frac{\pi}{4}\right)+\mathrm{e}^{-s_2}\right)\right.\\
		&\quad\left.-\sqrt{\frac{2}{\pi s_1}}\alpha_J(s_2)\left(\sin\left(s_1-\frac{\pi}{4}\right)+\mathrm{e}^{-s_1}\right)\right.\\
		&\quad\left.+\sqrt{\frac{2}{\pi s_1}}(\alpha_Y(s_2)+\alpha_H(s_2))\cos\left(s_1-\frac{\pi}{4}\right)\right.\\
		&\quad\left.-\sqrt{\frac{2}{\pi s_2}}\left(\alpha_Y(s_1)+\alpha_H(s_1)\right)\cos\left(s_2-\frac{\pi}{4}\right)\right\}.
	\end{align*}
	
	Using \cref{eq:3.7} and \cref{eq:alpha_l_bound}, we derive that
	\begin{align*}
		\left|\gamma(s_1,s_2)\right|&\le\frac{1}{k^4}\left(\frac{3}{\pi}s_1^{-3/2}s_2^{-1/2}+\frac{3}{\pi}s_1^{-1/2}s_2^{-3/2}+\frac{1}{8\pi}s_1^{-3/2}s_2^{-3/2}\right)\\
		&=\frac{s_1^{-1/2}s_2^{-1/2}}{\pi k^4}\left(3s_1^{-1}+3s_2^{-1}+\frac 1 8 s_1^{-1}s_2^{-1}\right)\\
		&\le\frac{1}{\pi k^5\sqrt{r_{j,1}r_{j,2}}}\left(\frac{6}{\tau\pi}+\frac{9}{\tau\pi}+\frac{3}{\tau^2\pi^2}\right)\\
		&=\frac{1}{\pi k^5\tau\sqrt{r_{j,1}r_{j,2}}}\left(\frac{6}{\pi}+\frac{9}{\pi}+\frac{3}{\pi^2}\right)\\
		&\approx\frac{1.2127}{\tau\pi k^5\sqrt{r_{j,1}r_{j,2}}}\\
		&\le\frac{3}{2\tau\pi k^5\sqrt{r_{j,1}r_{j,2}}}.
	\end{align*}
	This together with \eqref{eq:3.9}, $\lambda_{j,1}=1/2$ leads to
	\begin{align*}
		\left|\det(D_{j, k})\right|&\ge\frac{2}{\pi k^5\sqrt{r_{j,1}r_{j,2}}}\left|\sin\left(k(r_{j,1}-r_{j,2})\right)\right|-\left|\gamma(s_1,s_2)\right|\\
		&\ge\frac{1}{\pi k^5\sqrt{r_{j,1}r_{j,2}}}\left(2|\sin(k(r_{j,1}-r_{j,2}))|-\frac{3}{2\tau}\right)\\
		&\ge\frac{1}{\pi k^5\sqrt{r_{j,1}r_{j,2}}}\left(2|\sin(0.3363\pi)|-\frac{3}{2\tau}\right)\\
		&\ge \frac{1}{\pi k^5\sqrt{r_{j,1}r_{j,2}}}\left(1.74-\frac{3}{2\tau}\right)\\
		&\ge\frac{1}{\pi k^5R\sqrt{1+\lambda_{j,1}^2-2\lambda_{j,1}\cos\frac{\pi}{10}}}\left(1.74-\frac{3}{2\tau}\right)\\
		&\ge\frac{1}{1.717k^5R}\left(1.74-\frac{3}{2\tau}\right)\\
		&\ge\frac{M}{k^5R},
	\end{align*}
	with 
	\begin{equation}\label{eq: M}
		M=\frac{1}{2}\left(1.74-\frac{3}{2\tau}\right).
	\end{equation}
	
	Case (ii): $k=k_0=\frac{\pi}{30a}.$ We first recall that 
	\begin{align}\label{eq:J0}
		J_0(t)&=1-\frac{t^2}{4}+\widetilde{\alpha}_J(t),\\
		Y_0(t)&=\frac{2}{\pi}\left(1-\frac{t^2}{4}\right)\left(\ln\frac{t}{2}+C_0\right)+\frac{t^2}{2\pi}+\widetilde{\alpha}_Y(t),\label{eq:Y0}
	\end{align}
	with $C_0$ the Euler constant and
	\begin{align}\label{eq:alphaJY0}
		0<\widetilde{\alpha}_J(t)\le\frac{t^4}{64},\quad\left|\widetilde{\alpha}_Y(t)\right|\le\frac{t^3}{72}+\frac{t^4}{64},\quad0<t<2.
	\end{align}
	By letting $t_\ell=k_0r_{j,\ell}(\ell=1,2),$ we derive from \cref{eq:retrieve}, \cref{eq:J0}, and \cref{eq:Y0} that
	\begin{align}\nonumber
		\det\left(D_{j, k}\right)&=A_{j,1,k}B_{j,2,k}-A_{j,2,k}B_{j,1,k}\\ \nonumber
		&=\frac{1}{k_0^4}\Big(J_0(t_1)\left(Y_0(t_2)+H_0(t_2)\right)-J_0(t_2)\left(Y_0(t_1)+H_0(t_1)\right)\Big)\\\label{eq:detD}
		&=\frac{1}{k_0^4}\left(\frac{2}{\pi}\left(1-\frac{t_1^2}{4}\right)\left(1-\frac{t_2^2}{4}\right)\ln\frac{r_{j,2}}{r_{j,1}}+\frac{t_2^2-t_1^2}{2\pi}+\gamma(t_1,t_2)
		\right),
	\end{align}
	where 
	\begin{align*}
		\gamma(t_1,t_2)=\gamma_1(t_1,t_2)+\gamma_2(t_1,t_2),
	\end{align*}
	with 
	\begin{align*}
		\gamma_1(t_1,t_2)&=\widetilde{\alpha}_J(t_1)\left(\frac{2}{\pi}\left(1-\frac{t_2^2}{4}\right)\left(\ln\frac{t_2}{2}+C_0\right)+\frac{t_2^2}{2\pi}+\widetilde{\alpha}_Y(t_2)\right)\\
		&\quad-\widetilde{\alpha}_J(t_2)\left(\frac{2}{\pi}\left(1-\frac{t_1^2}{4}\right)\left(\ln\frac{t_1}{2}+C_0\right)+\frac{t_1^2}{2\pi}+\widetilde{\alpha}_Y(t_1)
		\right)\\
		&\quad+\left(1-\frac{t_1^2}{4}\right)\widetilde{\alpha}_Y(t_2)-\left(1-\frac{t_2^2}{4}\right)\widetilde{\alpha}_Y(t_1),
	\end{align*}
	and 
	\begin{align*}
		\gamma_2(t_1,t_2)=J_0(t_1)H_0(t_2)-J_0(t_2)H_0(t_1).
	\end{align*}
	Using \eqref{eq:alphaJY0}, $k_0R=\pi/5,$ $0.5k_0R\le t_1\le 0.55k_0R,$ and $2.47k_0R\le t_2\le 2.5k_0R,$ it has been proven in \cite{IP18} that 
	$$
	|\gamma_1(t_1,t_2)|<0.2255.
	$$
	Concerning the estimate on $|\gamma_2(t_1,t_2)|,$ we notice that $H_0(t)$ and $J_0(t)$ are continuous with respective to $t.$ Then for $(t_1,t_2)\in[0.5k_0R,0.55k_0R]\times[2.47k_0R,2.5k_0R],$ 
	$$
	|\gamma_2(t_1,t_2)|=\left|J_0(t_1)H_0(t_2)-J_0(t_2)H_0(t_1)\right|\le 0.2847.
	$$
	Combining \eqref{eq:detD}  and the estimates for $|\gamma_1(t_1,t_2)|$ and $|\gamma_2(t_1,t_2)|$ gives
	\begin{align*}
		\left|\det(D_{j, k})\right|&\ge\frac{1}{k_0^4}\left(
		\frac{2}{\pi}\left(1-\frac{t_1^2}{4}\right)\left(1-\frac{t_2^2}{4}\right)\ln\frac{r_{j,2}}{r_{j,1}}+\frac{t_2^2-t_1^2}{2\pi}-\gamma(t_1,t_2)
		\right)\\
		&\ge\frac{1}{k_0^4}\left(0.3554+0.3643-0.2255-0.2847\right)\\
		&>\frac{3}{20k_0^4},
	\end{align*}
	which completes the proof for $\left|\det(D_{j, k})\right|$.
	
	Summarizing the cases (i) and (ii), and letting
	$$
	C:=\min\left\{\frac{1}{2R}\left(1.74-\frac{3}{2\tau}\right), \frac{3k_0}{20}\right\},
	$$
	we obtain the estimate \cref{eq: bound} for $\square=\mathcal{I}$. The estimate for $|\det(D_{j, k}^\Delta)|$ can be deduced similarly.
\end{proof}

The above theorem sheds light on the unique solvability of \cref{eq: equation}, which relies on the parameters chosen in \cref{eq: parameter}. We would like to point out that there may exist other strategies for choosing the parameters in \cref{eq: parameter}.

Once \cref{eq: equation} is solvable, we are now heading for the stability of the phase retrieval formula. For fixed $k$ and $j,$ consider the following perturbed equation system for $\Re(u_{j, k}^\varepsilon)$ and $\Im(u_{j, k}^\varepsilon):$
\begin{equation}\label{eq:disturb}
	\begin{cases}
		A_{j,1,k}\Re(u_{j, k}^\varepsilon)+B_{j,1,k}\Im(u_{j, k}^\varepsilon)=f_{j,1,k}^\varepsilon, \\
		A_{j,2,k}\Re(u_{j, k}^\varepsilon)+B_{j,2,k}\Im(u_{j, k}^\varepsilon)=f_{j,2,k}^\varepsilon,
	\end{cases}
\end{equation}
with
$$
f_{j,\ell, k}^\varepsilon=\frac{4}{c_{j,\ell, k}^\varepsilon}\left(\left|v^\varepsilon_{j,\ell, k}\right|^2-\left|u_{j, k}^\varepsilon\right|^2\right)-\frac{c_{j,\ell, k}^\varepsilon}{16}\left(\left|A_{j,\ell, k}\right|^2+\left|B_{j,\ell, k}\right|^2\right)
$$
and
$$
c_{j,\ell, k}^\varepsilon=\frac{\|u^\varepsilon(\cdot, k)\|_{j,\infty}}{\|\Phi_k(\cdot, z_{j,\ell})\|_{j,\infty}}.
$$
Here, the measured noisy data $\left|u_{j, k}^\varepsilon\right|$ and $\left|v_{j,\ell, k}^\varepsilon\right|(\ell=1,2)$ satisfies
\begin{align*}
	\||u_{j, k}^\varepsilon|-|u_{j, k}|\|_{j,\infty} & \le\varepsilon\|u_{j, k}\|_{j,\infty},\\
	\left\|\left|v_{j,\ell, k}^\varepsilon\right|-\left|\hat{v}_{j,\ell, k}\right|\right\|_{j,\infty} & \le\varepsilon\|\hat{v}_{j,\ell, k}\|_{j,\infty},
\end{align*}
where $\varepsilon\in(0,1),\,\ell=1,2,$ and $\hat{v}_{j,\ell, k}(\cdot)=u_{j, k}(\cdot)-c_{j,\ell, k}^\varepsilon\Phi_k(\cdot, z_{j,\ell}).$

We now present the stability results as follows.
\begin{theorem}\label{thm: stability}
	Under \eqref{eq: parameter}, there exists a constant $C_{\varepsilon}$, such that:
	$$
	\left\|u_{j, k}^\varepsilon-u_{j, k}\right\|_{j,\infty}\le C_\varepsilon\|u_{j, k}\|_{j,\infty},\quad j=1,\cdots, m.
	$$
\end{theorem}

\begin{proof}
	Following the analysis conducted in \cite[Theorem 3.2]{IP18}, we know that 
	\begin{align*}
		&|c_{j,\ell, k}^\varepsilon-c_{j,\ell, k}|\le \varepsilon c_{j,\ell, k},\\ 
		&\left\|\left|v_{j,\ell, k}^\varepsilon\right|-\left|v_{j,\ell, k}\right|\right\|_{j,\infty}\le\varepsilon(3+\varepsilon)\|u_{j, k}\|_{j,\infty},\\
		&\left|\frac{\left|v_{j,\ell, k}^\varepsilon\right|^2}{c_{j,\ell, k}^\varepsilon}-\frac{\left|v_{j,\ell, k}\right|^2}{c_{j,\ell, k}}\right|\le\frac{\varepsilon\left((3+\varepsilon)(2+\varepsilon)^2+2\right)}{1-\varepsilon}\|u_{j, k}\|_{j,\infty}\|\Phi_k(\cdot, z_{j,\ell})\|_{j,\infty},\\
		&\left|\frac{\left|u_{j, k}^\varepsilon\right|^2}{c_{j,\ell, k}^\varepsilon}-\frac{\left|u_{j, k}\right|^2}{c_{j,\ell, k}}\right|\le\frac{\varepsilon(3+\varepsilon)}{1-\varepsilon}\|u_{j, k}\|_{j,\infty}\|\Phi_k(\cdot, z_{j,\ell})\|_{j,\infty}
	\end{align*}
	From the definitions of \eqref{eq:A} and \eqref{eq:f}, we find that
	\begin{align}\nonumber
		\left|f_{j, \ell, k}^\varepsilon-f_{j,\ell, k}\right| & \le 4\left|\frac{\left|v_{j,\ell, k}^\varepsilon\right|^2}{\left|c_{j,\ell, k}^\varepsilon\right|}-\frac{\left|v_{j,\ell, k}\right|^2}{\left|c_{j,\ell, k}\right|}\right|+4\left|\frac{\left|u_{j, k}^\varepsilon\right|^2}{\left|c_{j,\ell, k}^\varepsilon\right|}-\frac{\left|u_{j, k}\right|^2}{\left|c_{j,\ell, k}\right|}\right|\\
		&\quad+\frac{\left|c_{j,\ell, k}^\varepsilon-c_{j,\ell, k}\right|}{16}\left(A_{j,\ell, k}^2+B_{j,\ell, k}^2\right)\nonumber\\
		&\le\varepsilon\eta_\varepsilon\|u_{j, k}\|_{j,\infty}\|\Phi_k(\cdot, z_{j,\ell})\|_{j, \infty},\label{eq:f_error}
	\end{align}
	where 
	$$
	\eta_\varepsilon=\frac{1}{1-\varepsilon}\left(4\left(3+\varepsilon\right)\left(2+\varepsilon\right)^2+4(3+\varepsilon)+8\right)+4 (c_{j,1,k}+c_{j,2,k}).
	$$
	The following analysis is divided into two cases, depending on the value of $k.$
	
	Case (i): $k\in\mathbb{K}_N\backslash\{k_0\}.$ From \cite[(3.23)]{IP18}, we know that
	\begin{align}\label{eq:Y}
		\left|J_0(k r_{j,\ell})\right|\le\frac{0.82}{\sqrt{kR(1-\lambda_{j,\ell})}},\quad
		\left|Y_0(k r_{j,\ell})\right|\le\frac{0.82}{\sqrt{kR(1-\lambda_{j,\ell})}}.
	\end{align}
	From \eqref{eq:errorH0}, it holds that 
	$$
	\left|H_0(k r_{j,\ell})\right|\le\frac{3}{\sqrt{2\pi}}k^{-3/2}r_{j,\ell}^{-3/2}
	\le\frac{3}{(2\pi)^{3/2}\sqrt{kR(1-\lambda_{j,\ell})}}
	\le\frac{0.2}{\sqrt{kR(1-\lambda_{j,\ell})}}.
	$$
	This together with \eqref{eq:Y} gives
	$$
	\left|Y_0(k r_{j,\ell})+H_0(k r_{j,\ell})\right|\le\frac{1.02}{\sqrt{kR(1-\lambda_{j,\ell})}}.
	$$
	Noticing \eqref{eq:A}, we derive that
	\begin{align*}
		|A_{j,\ell, k}|\le\frac{1.02}{k^2\sqrt{kR(1-\lambda_{j,\ell})}},\ \
		|B_{j,\ell, k}|\le\frac{0.82}{k^2\sqrt{kR(1-\lambda_{j,\ell})}},
	\end{align*}
	which, together with \eqref{eq:f_error} and \eqref{eq:Y} gives
	\begin{align*}
		&\quad\left|B_{j, 2, k}\left(f_{j,1,k}^\varepsilon-f_{j,1,k}\right)-B_{j,1,k}\left(f_{j,2,k}^\varepsilon-f_{j,2,k}\right)\right|\\
		&\le\frac{0.82\varepsilon\eta_\varepsilon\|u_{j, k}\|_{j,\infty}}{k^2}
		\left(
		\frac{\|\Phi_k(\cdot, z_{j,1})\|_{j,\infty}}{\sqrt{kR(1-\lambda_{j,2})}}
		+\frac{\|\Phi_k(\cdot, z_{j,2})\|_{j,\infty}}{\sqrt{kR(1-\lambda_{j,1})}}
		\right) \\
		&\le\frac{1.64\varepsilon\eta_\varepsilon\|u_{j, k}\|_{j,\infty}}{5k^4\sqrt{kR(1-\lambda_{j,1})}\sqrt{kR(1-\lambda_{j,2})}}\\
		&\le\frac{\varepsilon\eta_\varepsilon}{k^5R}\|u_{j, k}\|_{j, \infty}.
	\end{align*}
	
	Analogously, we can prove that 
	$$
	\left|A_{j, 2, k}\left(f_{j,1,k}^\varepsilon-f_{j,1,k}\right)-A_{j,1,k}\left(f_{j,2,k}^\varepsilon-f_{j,2,k}\right)\right|\le\frac{\varepsilon\eta_\varepsilon}{k^5R}\|u_{j, k}\|_{j,\infty}.
	$$
	Subtracting \eqref{eq: equation} from \eqref{eq:disturb} gives 
	\begin{equation}\label{eq:error}
		\begin{cases}
			A_{j,1,k}\Re(u_{j, k}^\varepsilon-u_{j, k})+B_{j,1,k}\Im(u_{j, k}^\varepsilon-u_{j, k})=f_{j,1,k}^\varepsilon-f_{j,1,k},\\
			A_{j,2,k}\Re(u_{j, k}^\varepsilon-u_{j, k})+B_{j,2,k}\Im(u_{j, k}^\varepsilon-u_{j, k})=f_{j,2,k}^\varepsilon-f_{j,2,k},
		\end{cases}
	\end{equation}
	which implies that
	\begin{align*}
		&\left|\Re(u_{j, k}^\varepsilon-u_{j, k})\right|=\frac{\left|B_{j,2,k}\left(f_{j,1,k}^\varepsilon-f_{j,1,k}\right)-B_{j,1,k}\left(f_{j,2,k}^\varepsilon-f_{j,2,k}\right)\right|}{\left|\det(D_{j, k})\right|}\le\frac{\varepsilon\eta_\varepsilon}{M}\|u_{j, k}\|_{j,\infty}\\
		&\left|\Im(u_{j, k}^\varepsilon-u_{j, k})\right|=\frac{\left|A_{j,1,k}\left(f_{j,2,k}^\varepsilon-f_{j,2,k}\right)-A_{j,2,k}\left(f_{j,1,k}^\varepsilon-f_{j,1,k}\right)\right|}{\left|\det(D_{j, k})\right|}\le\frac{\varepsilon\eta_\varepsilon}{M}\|u_{j, k}\|_{j,\infty}.
	\end{align*}
	where $M$ is given by \cref{eq: M}. Thus,
	\begin{align}
		\|u_{j, k}^\varepsilon-u_{j, k}\|\le C_\varepsilon\|u_{j, k}\|_{j,\infty},
	\end{align}
	with 
	\begin{equation}\label{eq: c_delta}
		C_\varepsilon=\frac{\sqrt{2}\varepsilon\eta_\varepsilon}{M}.
	\end{equation}
	
	Case (ii): $k=k_0.$ From \cite[(3.28),(3.29)]{IP18}, $0.314\le t_1\le 0.346,$ $1.551\le t_2\le 1.571,$ we know that
	\begin{align}\label{eq:JY0}
		&\left|J_0(t_\ell)\right|<\left\{
		\begin{aligned}
			0.976,&&\ell=1\\
			0.49,&&\ell=2
		\end{aligned}
		\right.,\quad
		\left|Y_0(t_\ell)\right|<\left\{
		\begin{aligned}
			0.777,&&\ell=1\\
			0.413,&&\ell=2
		\end{aligned}
		\right.,\\\label{eq:H00}
		&\left|H_0(t_\ell)\right|<\left\{
		\begin{aligned}
			0.848,&&\ell=1\\
			0.128,&&\ell=2
		\end{aligned}
		\right. 
		.
	\end{align}
	Thus, from \eqref{eq:f_error}, \eqref{eq:JY0}, and \eqref{eq:H00}, we deduce that
	\begin{align*} 
		&\quad\left|B_{j,2,k}\left(f_{j,1,k}^\varepsilon-f_{j,1,k}\right)-B_{j,1,k}\left(f_{j,2,k}^\varepsilon-f_{j,2,k}\right)\right|\\
		&\le\frac{1}{k_0^2}\left(
		0.49\varepsilon\eta_\varepsilon\|u_{j, k}\|_{j, \infty}\|\Phi_k(\cdot, z_{j,1})\|_{j, \infty}+0.976\varepsilon\eta_\varepsilon\|u_{j, k}\|_{j,\infty}\|\Phi_k(\cdot, z_{j,2})\|_{j,\infty}\right)\\
		&\le\frac{0.21}{k_0^4}\varepsilon\eta_\varepsilon\|u_{j, k}\|_{j, \infty},
	\end{align*}
	and
	\begin{align*}
		&\quad\left|A_{j,2,k}\left(f_{j,1,k}^\varepsilon-f_{j,1,k}\right)-A_{j,1,k}\left(f_{j,2,k}^\varepsilon-f_{j,2,k}\right)\right|\\
		&\le\frac{1}{k_0^2}\left(
		0.543\varepsilon\eta_\varepsilon\|u_{j, k}\|_{j,\infty}\|\Phi_k(\cdot, z_{j,1})\|_{j, \infty}+1.627\varepsilon\eta_\varepsilon\|u_{j, k}\|_{j,\infty}\|\Phi_k(\cdot, z_{j,2})\|_{j,\infty}\right)\\
		&\le\frac{0.28}{k_0^4}\varepsilon\eta_\varepsilon\|u_{j, k}\|_{j,\infty}.
	\end{align*}
	Combined with \eqref{eq: bound}, this leads to the conclusion.
\end{proof}

We end this section with a brief description of the stability result regarding $\Delta u_{j, k}.$ The proof can be obtained with a similar argument as analyzed in \cref{thm: stability}, so the proof is omitted.

\begin{theorem}
	Under \eqref{eq: parameter}, the following estimate holds:
	$$
	\left\|\Delta u_{j, k}^\varepsilon-\Delta u_{j, k}\right\|_{j,\infty}\le C_\varepsilon\|\Delta u_{j, k}\|_{j,\infty},\quad j=1,\cdots, m.
	$$
	where $C_\varepsilon$ is given by \cref{eq: c_delta}.
\end{theorem}


\section{Numerical experiments}\label{sec: numerical}

In this section, we shall conduct several numerical experiments to verify the performance of the proposed method. 

In the experiments, we solve the forward problem via direct integration, and the unique solution to the biharmonic equation \cref{eq: biharmonic}--\cref{eq:Sommerfeld} is given by \cref{eq: solution}. Correspondingly, $\Delta u(x, k)$ can be written as
$$
\Delta u(x, k)=\int_{V_0}\Delta_x\Phi_k(x, y)S(y)\,\mathrm{d}y.
$$
Once the phaseless data $|u(x, k)|,\left|\Delta u(x, k)\right|$ is available, we add some uniformly distributed random noise to the measurements to test the stability of the algorithm. The following formula produces the noisy phaseless data:
$$
|u^\varepsilon|:=(1+\varepsilon r)|u|, \quad \left|\Delta u^\varepsilon\right|:=(1+\varepsilon r)|\Delta u|,
$$
where $r\in(-1,1)$ is a uniformly distributed random number, and $\varepsilon>0$ is the noise level.

In the numerical experiments, the test source function is chosen to be
\begin{align*}
	S(x_1,x_2) & =0.3(1-x_1)^2\mathrm{e}^{-x_1^2-(x_2+1)^2}-(0.2x_1-x_1^3-x_2^5)\mathrm{e}^{-x_1^2-x_2^2}
	-0.03\mathrm{e}^{-(x_1+1)^2-x_2^2}.
\end{align*}
The surface and contour plots of the exact source function $S$ in the domain $[-3,3]\times[-3,3]$ are plotted in \Cref{fig:exactsource}.

\begin{figure}[htp]
	\centering
	\includegraphics[width=.45\linewidth]{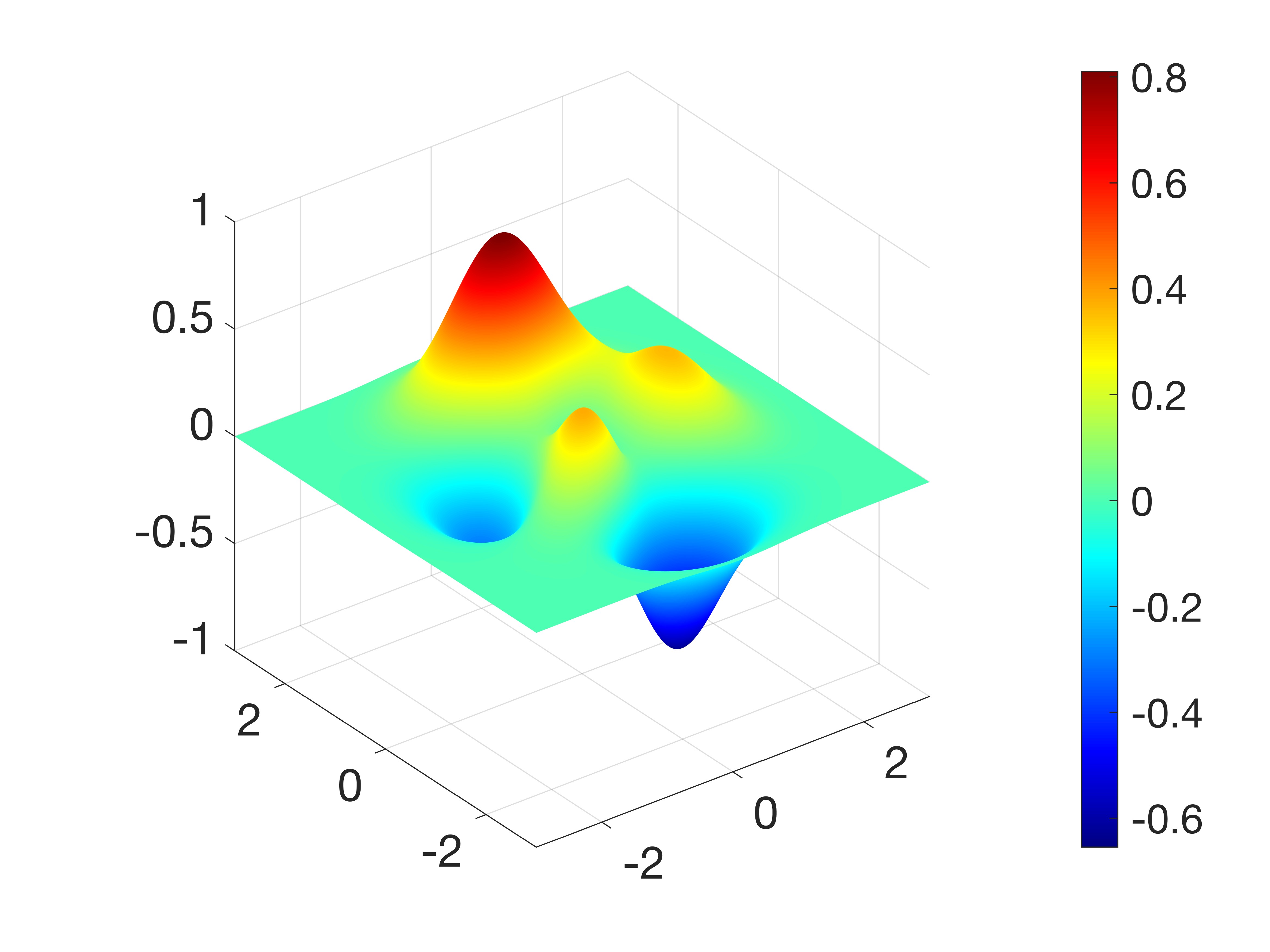}
	\includegraphics[width=.45\linewidth]{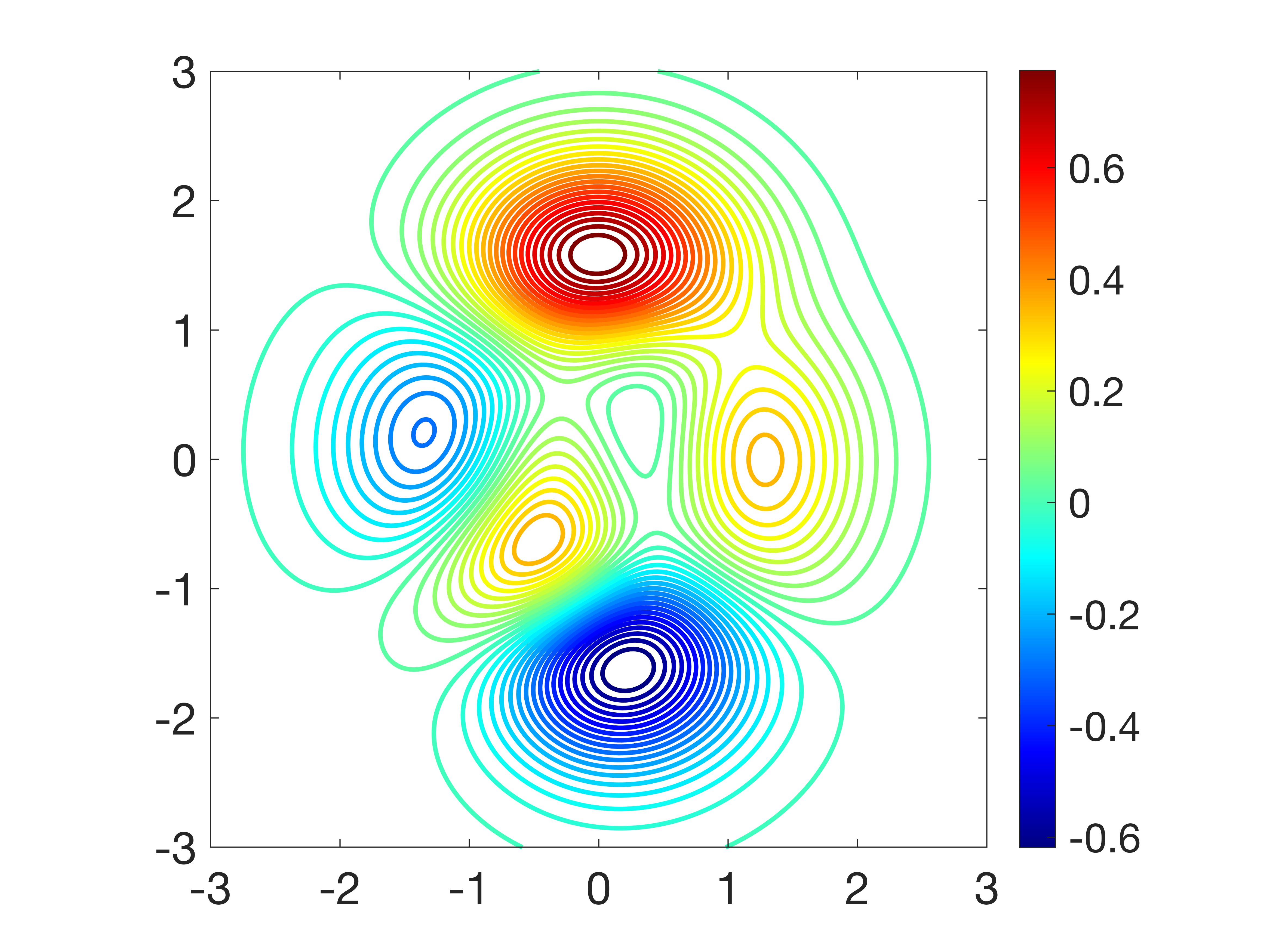}
	\caption{The exact source function $S$. (a) Surface plot; (b) contour plot.}
	\label{fig:exactsource}
\end{figure}

The rest of this section is divided into two parts. In \Cref{sec: phase}, we verify the performance of the phase retrieval algorithm proposed in \Cref{alg: PR}. In \Cref{sec:source}, the Fourier method described in \Cref{alg: Fourier} is adopted to reconstruct the source function $S.$

\subsection{Phase retrieval}\label{sec: phase}

In this subsection, we utilize \Cref{alg: PR} to retrieve the phase information from the measured phaseless radiating data.

We first specify the parameters in \cref{eq: parameter} by taking $a=3, \tau=6, m=10,$ and $k_0=\pi/90.$ In addition, we also consider the phase retrieval by taking the frequencies to be $k={\pi}/{3},\,{5\pi}/{3},$ and ${10\pi}/{3}$, respectively. We take the observation radius to be $R=6 a=18$. For a clear visualization of the model setup, we refer to \Cref{fig: model}, where we choose two typical wave numbers to show the geometry setup. Note that the locations of the artificial source points may differ under different wave numbers. 

\begin{figure}[htp]
	\centering
	\includegraphics[width=.45\linewidth]{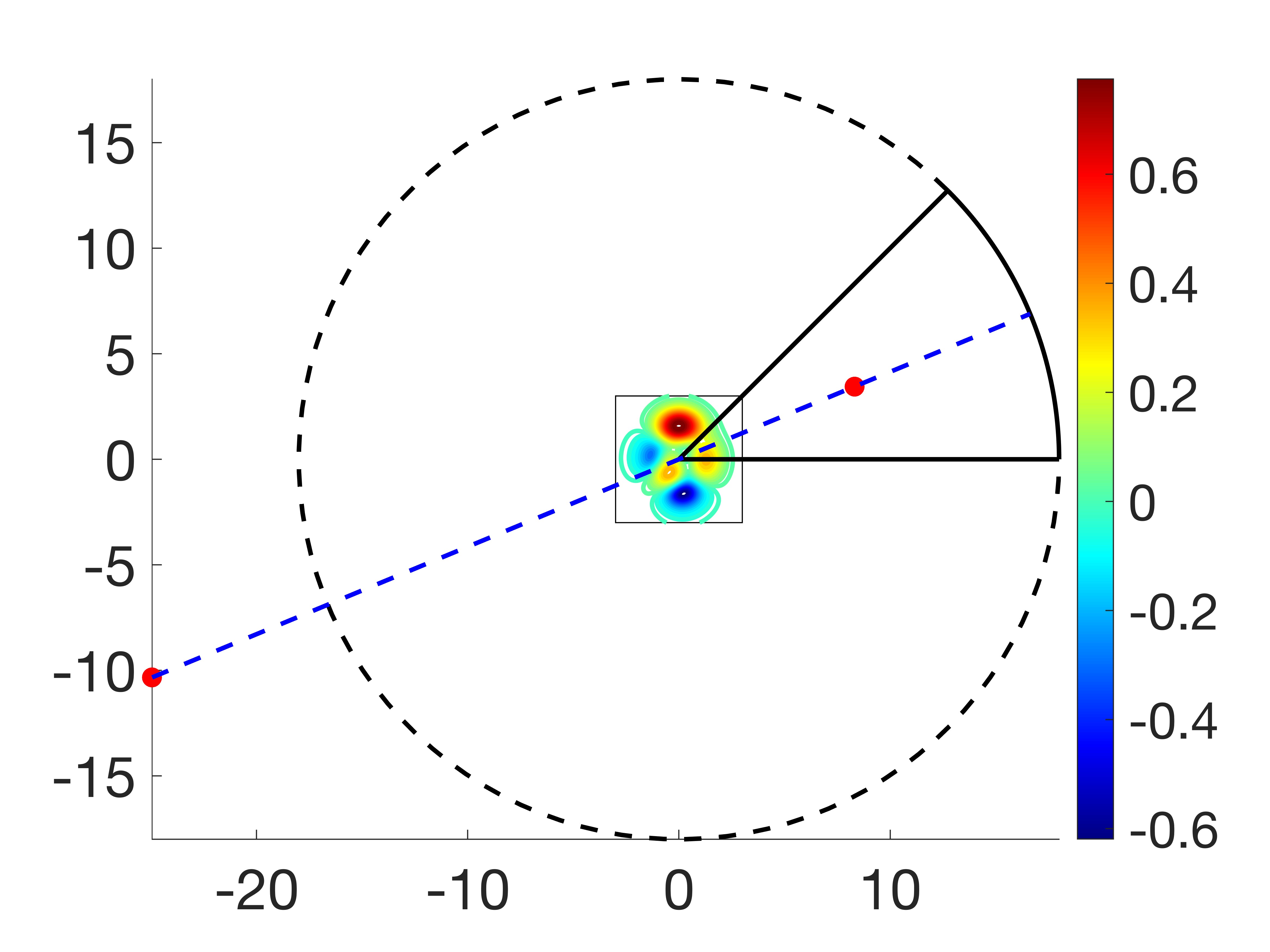}
	\includegraphics[width=.45\linewidth]{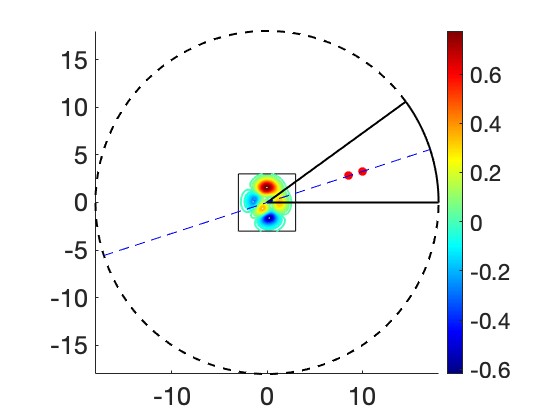}
	\caption{Geometry setup for the phase retrieval technique. The thick black line plots the boundary of $B_1$. The small red points denote the reference source points. (a) $k=k_0=\pi/90;$ (b) $k=\pi.$}
	\label{fig: model}
\end{figure}

For a quantitative evaluation of the accuracy of the phase retrieval, we compare the relative errors between the exact- and reconstructed phase data in \Crefrange{tab:erroruL2}{tab:errorDeltauInf}, where the discrete relative $L^2$ and infinity errors are respectively defined by 
\begin{align*}
	&\frac{\left(\sum\limits_{n=1}^{N_j}\left|w(x_n, k)-w^\varepsilon(x_n, k)\right|^2\right)^{1/2}}{\left(\sum\limits_{n=1}^{N_j}\left|w(x_n, k)\right|^2\right)^{1/2}},\quad x_n\in\Gamma_j,\quad j=1,\cdots, m, \\
	&\frac{\max\limits_{x_n\in\Gamma_j}\left|w(x_n, k)-w^\varepsilon(x_n, k)\right|}{\max\limits_{x_n\in\Gamma_j}\left|w(x_n, k)\right|},\quad x_n\in\Gamma_j,\quad j=1,\cdots, m.
\end{align*}
Here, $N_j$ is the number of the measurement points located on $\Gamma_j,$ $w$ and $w^\varepsilon$ designate the exact and retrieved data, respectively, with $w$ being $u$ or $\Delta u,$ depending on the measured phaseless field. To make the performance more reliable, we conducted 10 times of simulations, and all the listed outputs are the average results. From \Crefrange{tab:erroruL2}{tab:errorDeltauInf}, we find that the error of phase retrieval reduces when the noise level decreases, which is consistent with the theoretical analysis. Though the retrieval error is relatively large when the noise level increases, we can see in the next subsection that this discrepancy can be viewed as a perturbation to the radiated field, which would not substantially affect the final reconstruction of the source.

\begin{table}
	\renewcommand{\arraystretch}{1.2} 
	\centering
	\caption{The relative $L^2$ error between $u(\cdot, k)|_{\Gamma_1}$ and $u^\varepsilon(\cdot, k)|_{\Gamma_1}$}\label{tab:erroruL2}
	\begin{tabular}{lllll}
		\hline
		&$k_0=\frac{\pi}{90}$&$k=\frac{\pi}{3}$&$k=\frac{5\pi}{3}$&$k=\frac{10\pi}{3}$\\\hline
		$\varepsilon=0$&$6.58\times10^{-16}$&$6.45\times10^{-16}$&$3.18\times10^{-16}$&$5.20\times10^{-16}$\\
		$\varepsilon=0.1\%$&$0.24\%$&$0.26\%$&$0.19\%$&$0.20\%$\\
		$\varepsilon=1\%$&$3.66\%$&$2.56\%$&$1.93\%$&$1.48\%$\\
		$\varepsilon=5\%$&$15.11\%$&$13.54\%$&$12.59\%$&$10.57\%$ \\
		\hline
	\end{tabular}
\end{table}

\begin{table}
	\renewcommand{\arraystretch}{1.2} 
	\centering
	\caption{The relative $L^2$ error between $\Delta u(\cdot, k)|_{\Gamma_1}$ and $\Delta u^\varepsilon(\cdot, k)|_{\Gamma_1}$}\label{tab:errorDeltauL2}
	\begin{tabular}{lllll}
		\hline
		&$k_0=\frac{\pi}{90}$ & $k=\frac{\pi}{3}$ & $k=\frac{5\pi}{3}$ & $k=\frac{10\pi}{3}$ \\\hline
		$\varepsilon=0$&$3.49\times10^{-16}$&$4.21\times 10^{-16}$&$2.65\times10^{-16}$&$4.52\times10^{-16}$\\
		$\varepsilon=0.1\%$&$0.19\%$&$0.24\%$&$0.14\%$&$0.14\%$\\
		$\varepsilon=1\%$&$1.52\%$&$2.29\%$&$2.89\%$&$2.36\%$\\
		$\varepsilon=5\%$&$9.17\%$&$7.77\%$&$14.59\%$&$11.98\%$\\
		\hline
	\end{tabular}
\end{table}

\begin{table}
	\renewcommand{\arraystretch}{1.2} 
	\centering
	\caption{The relative infinity error between $u(\cdot, k)|_{\Gamma_1}$ and $u^\varepsilon(\cdot, k)|_{\Gamma_1}$}\label{tab:erroruInf}
	\begin{tabular}{lllll}\hline
		& $k_0=\frac{\pi}{90}$ & $k=\frac{\pi}{3}$ & $k=\frac{5\pi}{3}$ & $k=\frac{10\pi}{3}$\\\hline
		$\varepsilon=0$&$1.13\times10^{-15}$&$6.27\times10^{-16}$&$3.65\times10^{-16}$&$6.21\times10^{-16}$ \\
		$\varepsilon=0.1\%$&$0.33\%$&$0.27\%$&$0.23\%$&$0.25\%$ \\
		$\varepsilon=1\%$&$6.16\%$&$2.48\%$&$1.79\%$&$1.58\%$ \\
		$\varepsilon=5\%$&$19.32\%$&$17.16\%$&$18.10\%$&$13.72\%$ \\
		\hline
	\end{tabular}
\end{table}

\begin{table}
	\renewcommand{\arraystretch}{1.2} 
	\centering
	\caption{The relative infinity error between $\Delta u(\cdot, k)|_{\Gamma_1}$ and $\Delta u^\varepsilon(\cdot,k)|_{\Gamma_1}$}\label{tab:errorDeltauInf}
	\begin{tabular}{lllll}
		\hline
		& $k_0=\frac{\pi}{90}$ & $k=\frac{\pi}{3}$ & $k=\frac{5\pi}{3}$&$k=\frac{10\pi}{3}$\\\hline
		$\varepsilon=0$ & $6.22\times10^{-16}$ & $6.46\times10^{-16}$&$3.71\times10^{-16}$&$6.09\times10^{-16}$\\
		$\varepsilon=0.1\%$&$0.29\%$&$0.35\%$&$0.18\%$&$0.18\%$\\
		$\varepsilon=1\%$&$2.47\%$&$3.11\%$&$3.62\%$&$2.98\%$\\
		$\varepsilon=5\%$&$12.75\%$&$11.92\%$&$17.49\%$&$11.85\%$\\
		\hline
	\end{tabular}
\end{table}

\subsection{Reconstruct the source}\label{sec:source}
Once the phase is retrieved from the phaseless data, we can apply the Fourier method described in \Cref{sec: Fourier} to recover the source function from the multifrequency data.  Following \cite{CGYZ23}, we take the truncation of the Fourier expansion to be
$$
N=5\left[\varepsilon^{-1/4}\right],
$$
with $[X]$ being the largest integer that is smaller than $X+1.$ We use $400$ pseudo-uniformly distributed measurements on $\Gamma_R$ and $\Gamma_\rho$ with $\rho=20.$ We compute the relative $L^2$ error of the reconstruction over $V_0$ with $600\times600$ equidistantly distributed space grid. Concerning more implementation details, we refer to \cite{CGYZ23}.

\begin{figure}[htp]
	\centering 
	\includegraphics[width=.45\linewidth]{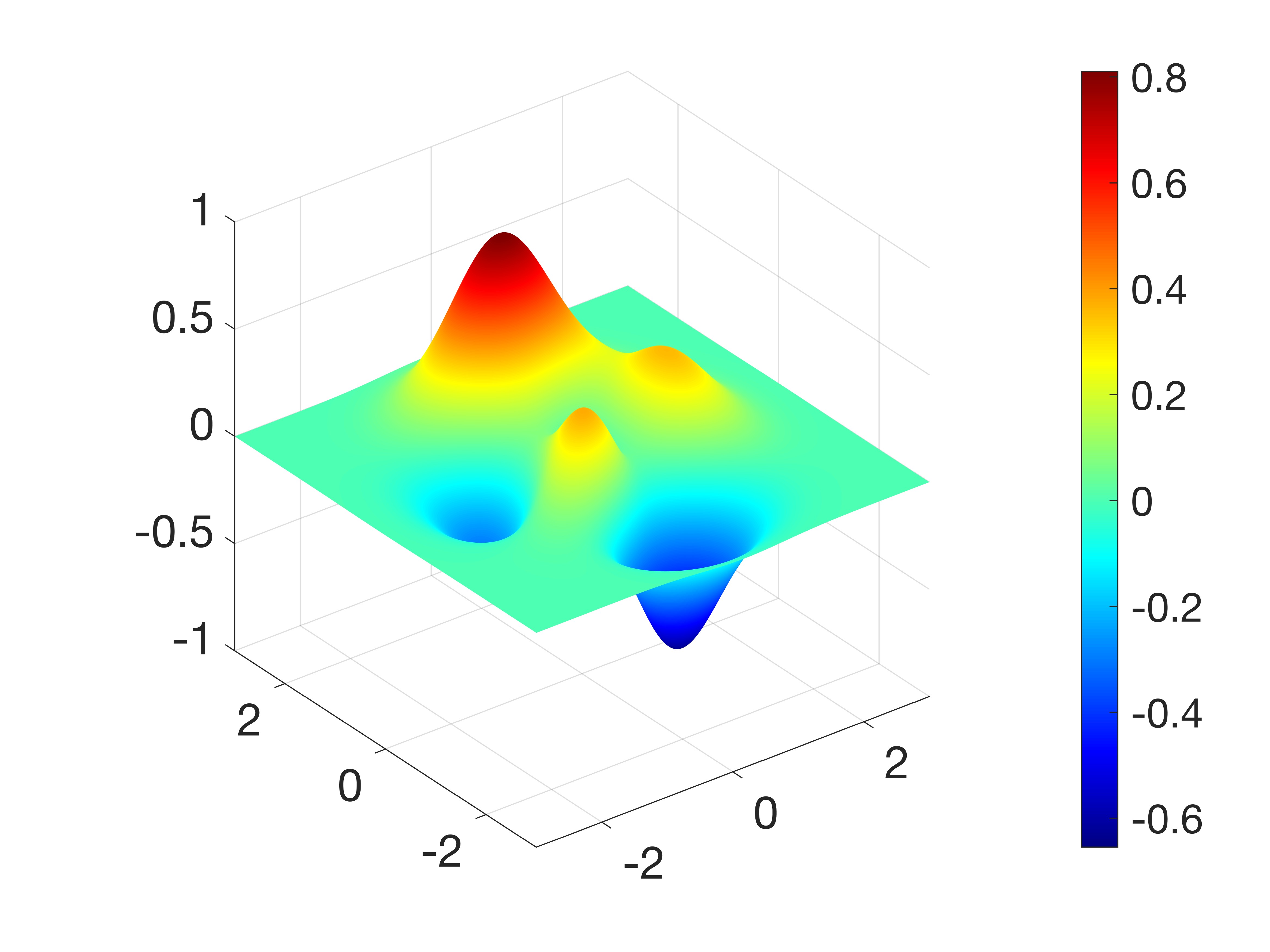}
	\includegraphics[width=.45\linewidth]{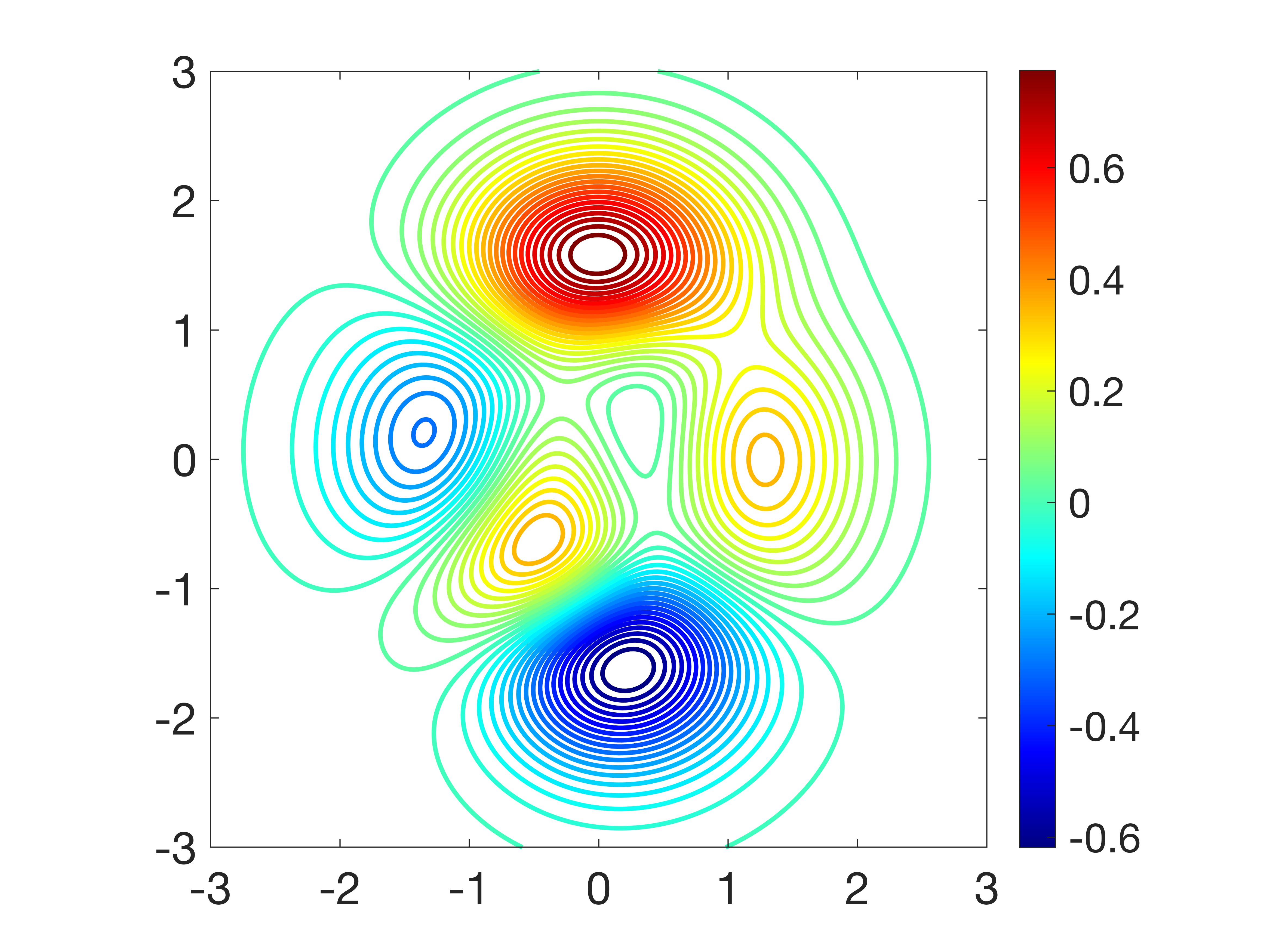}
	\caption{The reconstruction of source function $S$ with $\varepsilon=1\%$. (a) Surface plot; (b) contour plot.}
	\label{fig: reconstruction}
\end{figure}

The reconstruction with $1\%$ noise is illustrated in \Cref{fig: reconstruction}. Further, the relative $L^2$ errors with different noise levels are listed in \Cref{tab: error}.

\begin{table}
	\centering	
	\caption{The relative $L^2$ error between the reconstructed and the exact source function.}\label{tab: error}
	\begin{tabular}{ccccc}\hline
		$\varepsilon$ & $1\%$ & $5\%$ & $10\%$ & $20\%$\\\hline
		error & $1.26\%$ & $1.30\%$ & $1.37\%$ & $1.38\%$\\
		\hline
	\end{tabular}
\end{table}

From \Cref{tab: error}, it can be seen that the Fourier method is robust to the noise and has the capability of producing very high-resolution reconstructions.  Even if the noise level is up to $20\%,$ the final reconstruction error persists stunningly below $2\%.$ Recall that in Subsection \ref{sec: phase}, the largest phase retrieval error is $19.32\%.$ Although the intermediate phase retrieval step seems to be somehow (but reasonably) inaccurate with large noise, the eventual reconstructions are still satisfactory.


\begin{thebibliography}{99}
			
			\bibitem{Agaltsov} A. D. Agaltsov, T. Hohage, and R. G. Novikov, An iterative approach to monochromatic phaseless inverse scattering, {\em Inverse Problems}, {\bf 35} (2019), 024001. 
			
			\bibitem{Alvex} C. Alves, R. Kress, P. Serranho, Iterative and range test methods for an inverse source problem for acoustic waves, {\em Inverse Problems}, {\bf 25} (2009), 055005.
			
			\bibitem{Arfken} G. Arfken, B. Weber, J. Hans, Mathematical Methods for Physicists, Academic Press, 2013.
			
			\bibitem{Bao1} G Bao, J Lin, and Triki F, Numerical solution of the inverse source problem for the Helmholtz equation with multiple frequency data, {\em Contemp. Math. AMS}, {\bf 548} (2011), 45--60.
			
			\bibitem{Bao} G. Bao, S. Lu, W. Rundell, and B. Xu, A recursive algorithm for multifrequency acoustic inverse source problems, {\em SIAM J. Numer. Anal.}, {\bf 53} (2015), 1023--1035.
			
			\bibitem{BEEH} L. Beltrachini, N. von Ellenrieder, R. Eichardt, and J. Haueisen, Optimal design of on-scalp electromagnetic sensor arrays for brain source localisation, {\em Human Brain Mapping}, {\bf 42} (2021), 4859--4879.
			
			\bibitem{9} E. J. Cand\'{e}s, X. Li, and M. Soltanolkotabi, Phase retrieval via Wirtinger flow: theory and algorithms, {\em IEEE Trans. Inform. Theory}, {\bf 61} (2015), 1985--2007. 
			
			\bibitem{CGYZ23} Y. Chang, Y. Guo, T. Yin, and Y. Zhao, Mathematical and numerical study of an inverse source problem for the biharmonic wave equation, {\em arXiv:2307.02072}
			
			\bibitem{11}  X. Chen, Computational Methods for Electromagnetic Inverse Scattering, Wiley, New York, 2018.
			
			\bibitem{Deng} Y. Deng, H. Liu, and G. Uhlmann, On an inverse boundary problem arising in brain imaging, {\em J. Differential Equations}, {\bf267} (2019), 2471--2502.
			
			\bibitem{Dong} H. Dong, D. Zhang, and Y. Guo, A reference ball based iterative algorithm for imaging acoustic obstacle from phaseless far-field data, {\em Inverse Problems}, {\bf 13} (2019), 177--195.
			
			\bibitem{8} M. Farhat, S. Guenneau, and S. Enoch, Ultrabroadband elastic cloaking in thin plates, {\em Phys. Rev. Lett.}, {\bf 103} (2009), 024301.
			
			
			\bibitem{47} M.H. Maleki, and A.J. Devaney, Phase-retrieval and intensity-only reconstruction algorithms for optical diffraction tomography, {\em J. Opt. Soc. Am. A}, {\bf 10} (1993), 1086--1092.	
			
			\bibitem{48}  S. Maretzke and T. Hohage, Stability estimates for linearized near-field phase retrieval in X-ray phase contrast imaging, {\em SIAM J. Appl. Math.}, {\bf 77} (2017), 384--408.
			
			\bibitem{22} R.C. McPhedran, A.B. Movchan, and N.V. Movchan, Platonic crystals: Bloch bands, neutrality and defects, Mech. Mater., {\bf 41} (2009), 356--363.
			
			\bibitem{Leone} G. Leone, M.A. Maisto, and R. Pierri, Application of inverse source reconstruction to conformal antennas synthesis, {\em IEEE Transactions on Antennas and Propagation.}, {\bf 66}(2018), 1436--1445.
			
			\bibitem{LYZZ21} L. Li, J. Yang, B. Zhang, and H. Zhang, Imaging of buried obstacles in a two-layered medium with phaseless far-field data. {\em Inverse Problems}, {\bf 37}(2021), 055004
			
			\bibitem{LXD} X. Ji, X. Liu, and B. Zhang, Phaseless inverse source scattering problem: phase retrieval, uniqueness, and direct sampling methods, {\em Journal of Computational Physics: X.}, {\bf 1}(2019), 100003.
			
			\bibitem{elasticFourier} M. Song and X. Wang, A Fourier method to recover elastic sources with multi-frequency data, {\em East Asian J. Appl. Math.}, {\bf 9} (2019), 369--385.
			
			\bibitem{Stefanov} P. Stefanov and G. Uhlmann, Thermoacoustic tomography with variable sound speed, {\em Inverse Problems}. {\bf 25} (2009), 0775011.
			
			\bibitem{27} M. Stenger, M. Wilhelm, and M. Wegener, Experiments on elastic cloaking in thin plates, {\em Phys. Rev. Lett.}, {\bf 108} (2012), 014301.
			
			\bibitem{Thio} B. J. Thio, A. S. Aberra, G. E. Dessert, and W. M. Grill, Ideal current dipoles are appropriate source representations for simulating neurons for intracranial recordings, {\em Clinical Neurophysiology}, {\bf 145} (2023), 26--35
			
			\bibitem{WG} G. Wang, F. Ma, Y. Guo, and J. Li, Solving the multi-frequency electromagnetic inverse source problem by the Fourier method, {\em Journal of Differential Equations}, {\bf 265} (2018), 417--443
			
			\bibitem{WXCfar} X. Wang, Y. Guo, D. Zhang, and H. Liu, Fourier method for recovering acoustic sources from multi-frequency far-field data, {\em Inverse Problems}, {\bf 33} (2017), 035001
			
			\bibitem{30} E. Watanabe, T. Utsunomiya, and C. Wang, Hydroelastic analysis of pontoon-type VLFS: a literature survey, {\em Eng. Struct.}, {\bf 26} (2004), 245--256
			
			\bibitem{IP15} D. Zhang and Y. Guo, Fourier method for solving the multi-frequency inverse source problem for the Helmholtz equation, {\em Inverse Problems}, {\bf 31}(2015), 035007
			
			\bibitem{IP18} D. Zhang, Y. Guo. J. Li, and H. Liu, Retrieval of acoustic sources from multi-frequency phaseless data, {\em Inverse Problems}, {\bf 34}(2018), 094001
			
			\bibitem{IPI23}  D. Zhang, Y. Wu, and Y. Guo, Imaging an acoustic obstacle and its excitation sources from phaseless near-field data, {\em Inverse Problems and Imaging}, DOI: 10.3934/ipi.2023055
			
		\end{thebibliography}
	\end{document}